\documentclass{amsart}

\usepackage[skip=5pt plus1pt, indent=20pt]{parskip}

\usepackage[utf8]{inputenc}
\usepackage{amssymb} 
\usepackage{amsmath} 
\usepackage{amscd}
\usepackage{amsbsy}
\usepackage{comment}
\usepackage[matrix,arrow]{xy}
\usepackage{hyperref}
\usepackage{float}
\usepackage[utf8]{inputenc}
\usepackage{xcolor}
\usepackage{enumerate}
\usepackage{breqn}
\usepackage[all]{nowidow}

\DeclareMathOperator{\Norm}{Norm}

\DeclareMathOperator{\Rad}{Rad}

\DeclareMathOperator{\Gal}{Gal}

\DeclareMathOperator{\trace}{Tr}

\newcommand{\Q}{{\mathbb Q}}
\newcommand{\Z}{{\mathbb Z}}

\newcommand{\F}{{\mathbb F}}

\newcommand{\cO}{\mathcal{O}}

\def\mod#1{{\ifmmode\text{\rm\ (mod~$#1$)}
\else\discretionary{}{}{\hbox{ }}\rm(mod~$#1$)\fi}}

\begin{document}

\newtheorem{theorem}{Theorem}
\newtheorem{lemma}{Lemma}[section]
\newtheorem{proposition}[lemma]{Proposition}
\newtheorem{algorithm}[lemma]{Algorithm}
\newtheorem{corollary}[lemma]{Corollary}
\newtheorem*{conjecture}{Conjecture}

\theoremstyle{definition}
\newtheorem{definition}[theorem]{Definition}

\theoremstyle{remark}
\newtheorem{remark}[theorem]{Remark}

\newtheorem{acknowledgment}{Acknowledgement}

\title[]{Asymptotic Fermat's Last Theorem for a family of equations of signature $(2, 2n, n)$}

\author{Pedro-Jos\'{e} Cazorla Garc\'{i}a}

\address{Department of Mathematics, University of Manchester, Manchester, United Kingdom, M13 9PY}
\email{pedro-jose.cazorlagarcia@manchester.ac.uk}

\date{\today}

\begin{abstract}
In this paper, we study the integer solutions of a family of Fermat-type equations of signature $(2, 2n, n)$, $Cx^2 + q^ky^{2n} = z^n$. We provide an algorithmically testable set of conditions which, if satisfied, imply the existence of a constant $B_{C, q}$ such that if $n > B_{C,q}$, there are no solutions $(x, y, z, n)$ of the equation. Our methods use the modular method for Diophantine equations, along with level lowering and Galois theory.
\end{abstract}


\maketitle

\section{Introduction}


\subsection{Historical background}
At the beginning of the 17th century, Fermat wrote in the margin of a copy of \emph{Arithmetica} that he had proved that the exponential Diophantine equation
\begin{equation}
    \label{eqn:flt}
x^n+y^n = z^n
\end{equation}
had no solutions $(x, y, z, n) \in \Z^4$ with $n > 2$ and $xyz \neq 0$. 
Fermat's alleged proof of this fact was never found and the resolution of \eqref{eqn:flt} became one of the biggest problems in the history of mathematics, known as \emph{Fermat's Last Theorem}.

In 1995, Wiles \cite{Wiles} proved the Modularity Theorem for semistable elliptic curves, which, along with Ribet's Level Lowering Theorem \cite{Ribet}, finished the proof of Fermat's Last Theorem more than three centuries after its initial statement. 

After Wiles's proof of Fermat's Last Theorem, several generalisations of \eqref{eqn:flt} have been studied. For example, the equation
\begin{equation}
    \label{eqn:fermattype}
Ax^p + By^q = Cz^r,
\end{equation}
where $A$, $B$ and $C$ are given integers, is called a \emph{Fermat-type equation} of signature $(p, q, r)$. After the proof of Fermat's Last Theorem, many researchers (see \cite{BenS, inertia, BMS, Kraus}, among many others) have used the modular methodology pioneered by Wiles in order to study Fermat-type equations over $\Q$.

More recently, this methodology has also been used to study \eqref{eqn:fermattype} over number fields $K$. In this setting, we are interested in solutions where $x$, $y$ and $z$ are elements of the ring of integers of $K$, which we will denote by $\cO_K$. 

One of the most relevant results on Fermat-type equations over number fields is due to Freitas and Siksek \cite{asymptotic}. They showed that for $5/6$ of real quadratic fields $K = \Q(\sqrt{d})$, ordered by the value of $d$, there exists a constant $B_K$, depending only on $K$, such that if $p > B_K$ is prime, the equation 
\[x^p + y^p = z^p\]
has no solutions $(x, y, z) \in \cO_K^3$ satisfying $xyz \neq 0$. This is called the \emph{asymptotic Fermat's Last Theorem} (AFLT) \emph{for $K$}. For a general totally real number field $K$, they give an algorithmically testable criterion which, if satisfied, implies the Asymptotic Fermat's Last Theorem for $K$. Shortly afterwards, Deconinck \cite{coefficients} proved an analogous result for the equation
\[Ax^p + By^p = Cz^p,\]
where $A, B, C \in \cO_K$ are fixed and odd. In this situation, the constant implied by AFLT depends on $A$, $B$ and $C$, as well as on the number field $K$.

Other researchers have extended this line of work to Fermat-type equations of signatures $(p, p, 2)$ and $(p, p, 3)$. For example, I\c{s}ik, Kara and Özman \cite{pp2, pp3} have studied the Diophantine equations
\begin{equation}
    \label{eqn:pp2}
x^p + y^p = z^2, \quad 2 \mid y,
\end{equation}
and
\begin{equation}
    \label{eqn:pp3}
x^p + y^p = z^3, \quad 3 \mid y,
\end{equation}
over totally real number fields $K$ with narrow class number $h_K^+ = 1$. For these fields, they show that there is a constant $B_K$ such that, if $p > B_K$, there are no solutions $(x, y, z, p) \in \cO_K^3 \times \Z$ to either \eqref{eqn:pp2} or \eqref{eqn:pp3}. Finally, this work was extended by Mocanu \cite{Mocanu}, who showed the same results under a weaker assumption.

\subsection{The main results}
In this paper, we adapt these techniques to approach a different problem. We note that all the previously mentioned literature considers solutions of \textbf{one} exponential Diophantine equation over infinitely many number fields. However, we shall consider solutions to a family of \textbf{infinitely many} exponential Diophantine equations of signature $(2, 2n, n)$ over the rationals. The family of Fermat-type equations that we will consider is the following:
\begin{equation}
\label{eqn:main} 
Cx^2 + q^ky^{2n} = z^n, \quad \gcd(Cx, qy, z) = 1, \quad 2 \mid z,
\end{equation}
where $C, q$ and $k$ are fixed positive integers, with $C$ squarefree and $q \ge 3$ a prime. 
This equation is relevant because, to the best of our knowledge, there are no results on AFLT for Fermat-type equations of signature $(p,p,2)$ unless $C = 1$ or $C = 2$ (see \cite{Ivorra, IvorraKraus, Kumar1, Kumar2} for some of the existing results if $C = 1$ or $C = 2$), even over $\Q$. Since solutions of \eqref{eqn:main} are also solutions of
\[Cx^2 + q^ky^{n} = z^n, \quad \gcd(Cx, qy, z) = 1,\]
studying \eqref{eqn:main} gives information about a Fermat-type equation of signature $(n, n, 2)$ with $C \neq 1, 2$. 

In addition, we note that \eqref{eqn:main} is a generalisation of the Lebesgue--Nagell equation, which has been widely studied in the literature (for example, see \cite{BennettSiksek2, BMS, secondpaper} or \cite{survey} for an exposition of the history of the Lebesgue--Nagell equation) to three variables $x, y$ and $z$. 

We now proceed to give the definition of AFLT for \eqref{eqn:main}, which is the following:

\begin{definition}(AFLT for \eqref{eqn:main})
    \label{def:aflt}
     We say that Asymptotic Fermat's Last Theorem (AFLT) holds for \eqref{eqn:main} if there is a constant $B_{C, q}$, depending only on $C$ and $q$, such that if $p$ is a prime number with $p > B_{C, q}$, then there are no solutions $(x, y, z, n) \in \Z^4$ to \eqref{eqn:main} with $n=p$.
\end{definition}

We remark that Definition \ref{def:aflt} is stronger than some of the definitions of AFLT used in previous literature, in the sense that $B_{C, q}$ does not depend on $k$. 

\begin{remark}
    If the constant in Definition \ref{def:aflt} exists, then there is a different constant $B'_{C,q}$ such that if $n$ is composite and $n > B'_{C,q}$, there are no solutions $(x, y, z, n) \in \Z^4$ of \eqref{eqn:main}. Indeed, suppose that 
    Definition \ref{def:aflt} holds. Let
    \[D = \{4, 6, 9\} \cup \{p \ge 5 \text{ prime} \mid p < B_{C, q}.\}\]
    Any solution $(x, y, z, n) \in \Z^4$ of \eqref{eqn:main} with $n$ composite necessarily has $m \mid n$ for some $m \in D$, so we write $n = mt$ for some $t \ge 1$. The existence of such a solution $(x, y, z, n)$ of \eqref{eqn:main} clearly means that $(x', y', z') = (x, y^t, z^t)$ is a solution of 
    \begin{equation}
        \label{eqn:composite}
        C(x')^2 + q^k (y')^{2m} = (z')^m.
    \end{equation}
    A specialisation of a result of Darmon and Granville (see \cite[Theorem 2]{DarmonGran}) yields that there are finitely many solutions of \eqref{eqn:composite}
    for any $m \in D$. Then we can define 
    \[\begin{split} m' = \max\{s \ge 1 \mid \text{there exists } m \in D,  (x, y, z) \in \Z^3 \\ \text{ such that } (x, y^s, z^s) \text{ is a solution of } \eqref{eqn:composite}\} \cup \{1\}.
    \end{split}\]
    By our previous discussion, the set above is finite and, therefore, $m' < \infty$. By definition of $m'$, it follows that any solution $(x, y, z, n) \in \Z^4$ with $n$ composite satisfies
    \[n < \max\{9, B_{C, q}\}\cdot m',
    \]
    and so it suffices to take $B'_{C, q} = \max\{9, B_{C, q}\}\cdot m'$. However, we emphasise that the constant $m'$ can only be made effective by explicitly resolving \eqref{eqn:main}, which is beyond the scope of this paper. 
\end{remark}

We can now present the main result of this paper, Theorem \ref{thm:main}, which provides a set of algorithmically testable conditions which, if satisfied, imply AFLT for \eqref{eqn:main}.

\begin{theorem}
    \label{thm:main}
    Let $C \ge 1$ be a squarefree integer, $q \ge 3$ a prime number and $k \ge 0$ an integer and consider the following Diophantine equation:
    \begin{equation*}
        Cx^2 + q^k y^{2n} = z^n, \quad \gcd(Cx, qy, z) = 1, \quad 2 \mid z.
    \end{equation*}
    Suppose that the Diophantine equation
    \begin{equation}
    \label{eqn:mainobstruction}
    Ct^2 + q^\gamma = 2^m    
    \end{equation}
    has no solutions $(t, \gamma, m) \in \Z^3$ with $m > 6$ and $\gamma \ge 0$ with $\gamma \equiv k \pmod 2$. Suppose furthermore that any of the following hypotheses hold:
    \begin{enumerate}[(a)]
        \item The exponent $k$ is odd.
        \item {The exponent $k$ is even,} and $-C$ is not a square modulo $q$.
        \item {The exponent $k$ is even, $q \not \equiv 7 \pmod 8$ and} there are no solutions $(t, \gamma, m) \in \Z^3$ to the following equation:
        \begin{equation}
            \label{eqn:obstruction1}
            Ct^2 + 1 = q^\gamma2^m, \quad m > 6, \quad \gamma \ge 0.
        \end{equation}

        \item {The exponent $k$ is even, $q \equiv 7 \pmod 8$}, \eqref{eqn:obstruction1} has no solutions and there are no solutions $(t, \gamma, m) \in \Z^3$ to 
        \begin{equation}
            \label{eqn:obstruction2}
            Ct^2 + 2^m = q^\gamma, \quad m >6 \text{ even}, \quad \gamma > 0 \text{ odd}.
        \end{equation}
    \end{enumerate}
    Then AFLT holds for \eqref{eqn:main} and the constant $B_{C, q}$ can be explicitly computed.
\end{theorem}

We note that, in order to prove Theorem \ref{thm:main}, it suffices to do so under the assumption that $0 \le k < 2n$. Indeed, let us write $k = 2nk_1 + k_2$, with $k_1 \ge 0$ and $0 \le k_2 < 2n$. Then any solution $(x_0,y_0, z_0,n_0)$ of \eqref{eqn:main} gives rise to a solution $(x_0,y_0q^{k_1}, z_0, n_0)$ of 
\[
    Cx^2 + q^{k_2}y^{2n} = z^n,
\]
and, since $k \equiv k_2 \pmod{2}$, the conditions in Theorem \ref{thm:main} are well defined if we replace $k$ by $k_2$. For this reason, we shall assume that $0 \le k < 2n$ for the remainder of the paper.

In addition, we emphasise that the determination of whether any of the hy\-po\-the\-ses in Theorem \ref{thm:main} are satisfied can be done in a computationally effective manner. For this purpose, we provide the reader with \texttt{Magma} code in the GitHub repository \href{https://github.com/PJCazorla/Asymptotic-Fermat-s-Last-Theorem-for-a-family-of-equations-of-signature--n--2n--2-/tree/main}{https://shorturl.at/hoxW8}. We will explain the computations in Section \ref{Sec:computation}, allowing us to prove the following result.

\begin{theorem}
    \label{thm:computations}
    Let $1 \le C \le 70$ be a squarefree integer and $3 \le q < 100$ be a prime number. By reducing \eqref{eqn:main} modulo $8$, we see that $Cq^k \equiv 7 \pmod{8}$. Then AFLT holds for 268 out of the 330 pairs in this range. In addition, Table \ref{tab:computations} contains the number of pairs satisfying the conditions in Theorem \ref{thm:main}, as well as the total number of pairs for each value of $k\pmod{2}$.

    \begin{table}[!ht]
        \centering
        \begin{tabular}{|ccc|}
            \hline 
             $k \pmod{2}$ & \#Pairs $(C, q)$ & \# Pairs satisfying the conditions in Theorem \ref{thm:main}  \\
             \hline \hline 
             0 & 158 & 131 \\
             1 & 172 & 137 \\
             \hline \hline 
             \textbf{TOTAL} & 330 & 268 \\
             \hline 
        \end{tabular}
        \caption{Number of pairs satisfying the conditions in Theorem \ref{thm:main}.}
        \label{tab:computations}
    \end{table}
\end{theorem}





The structure of the paper is as follows. In Section \ref{Sec:preliminary}, we present the modular method for Diophantine equations and characterise under what conditions it fails to prove AFLT for \eqref{eqn:main}. In Sections \ref{Sec:inertia} and \ref{Sec:galois}, we will use ``image of inertia'' arguments and Galois theory respectively to build upon this characterisation. In Section \ref{Sec:possibilities}, we put the previous results together to show that the failure of AFLT implies the existence of an elliptic curve $E$ of a particular form. In Section \ref{Sec:theorem1}, we finish the proof of Theorem \ref{thm:main} by relating the curve $E$ to the existence of solutions to certain Diophantine equations. Finally, on Section \ref{Sec:computation}, we show that checking the conditions in Theorem \ref{thm:main} is a computationally effective process and prove Theorem \ref{thm:computations}.

\noindent \textbf{Acknowledgements} The author would like to thank Gareth Jones, Diana Mocanu and Lucas Villagra-Torcomián for their comments on a draft version of the paper and for useful conversations.

\section{The Frey--Hellegouarch curve and the modular method}
\label{Sec:preliminary}
In this section, we present the Frey--Hellegouarch that we shall use to achieve a bound on the exponent $n$. An excellent expository article on the modular method and its applications can be found in \cite{Siksek}. 

We highlight that, due to the fact that $z$ is even, there is only one Frey--Hellegouarch curve to consider, allowing for a uniform treatment of all cases. If $z$ were odd, the number of cases to consider would grow significantly and, consequently, we would not be able to get a result like Theorem \ref{thm:main}. We refer the reader to Remark \ref{rmk:toughluck} for a more detailed discussion on why this is the case.

We suppose that there exists a solution $(x, y, z, p)$ to \eqref{eqn:main} with $z$ even and $n = p \ge 7$ a prime number. Following Bennett and Skinner \cite{BenS}, we can associate the following elliptic curve to the solution:
\begin{equation}
        \label{eqn:frey}
        F = F(x, z, p): Y^2 + XY = X^3 + \frac{Cx-1}{4}X^2 + \frac{Cz^p}{64}X.
\end{equation}
We shall call $F$ the \emph{Frey--Hellegouarch curve associated to $(x, y, z, p)$}. By \cite[Lemma 2.1]{BenS}, the minimal dis\-cri\-mi\-nant of $F$ is given by
\begin{equation}
    \label{eqn:freyminimaldiscriminant}
    \Delta_F = -2^{-12}C^3q^k(yz)^{2p},
\end{equation}
and conductor given by
\begin{equation}
    \label{eqn:conductor}
N = \begin{cases}
    2C^2q\Rad_{2,q}(yz), & \text{if } k \neq 0, \\
    2C^2\Rad_{2}(yz), & \text{if } k = 0.
    \end{cases}
\end{equation}
where $\Rad_{2,q}(yz)$ denotes the product of all prime numbers dividing $yz$ except $2$ and $q$, and similarly for $\Rad_{2}(yz)$. By the Modularity Theorem \cite{Wiles}, the curve $F$ corresponds to a rational modular form of weight $2$ and level $N$. However, we remark that the level $N$ depends on our solutions and, therefore, is not explicit.

In order to be able to work with newforms of an explicit level, we will need to combine the Modularity Theorem with Ribet's Level Lowering Theorem \cite{Ribet}. We shall do so by applying \cite[Theorem 13]{Siksek} (which is a combination of \cite[Lemmas 3.2 and 3.3]{BenS}). By this result, it follows that either $yz = \pm 1$, or there exists a newform $f$ of weight $2$, trivial Nebentypus character and level given by
\begin{equation}
    \label{eqn:Np}
N' = \begin{cases}
        2C^2q & \text{if } k \neq 0,p, \\
        2C^2 & \text{if } k = 0,p,
      \end{cases}
\end{equation}
such that 
\begin{equation}
\label{eqn:galoisreps}
\overline{\rho}_p(F) \cong \overline{\rho}_p(f),
\end{equation}
where $\overline{\rho}_p$ denotes the mod-$p$ Galois representations associated to $F$ and $f$ respectively. We note that, since $z$ is even, $yz \neq \pm 1$. Consequently, \eqref{eqn:galoisreps} holds.

Given a prime number $\ell$, we define $a_\ell(F) := \ell+1-\#F(\mathbb{F}_\ell)$. Similarly, we let $c_\ell$ be the $\ell-$th coefficient in the Fourier cusp expansion of $f$, we let $K_f$ be the number field generated by all Fourier coefficients of $f$, and we let $\cO_{K_f}$ be its ring of integers. Then, by \cite[Propositions 5.1 and 5.2]{Siksek}, \eqref{eqn:galoisreps} is equivalent to
\begin{equation}
    \label{eqn:congruenceconditions}
    \begin{cases}
        a_\ell(F) \equiv c_\ell \text{ (mod } \mathfrak{p}\text{)} & \text{ if } \ell \neq p, \quad \ell \nmid N, \\
        c_\ell \equiv \pm(\ell+1) \text{ (mod } \mathfrak{p}\text{)} & \text{ if } \ell \neq p, \quad \ell \nmid N', \quad  \ell \mid N, \\
    \end{cases}
\end{equation}
where $\mathfrak{p}$ is some prime ideal of $\cO_{K_f}$ above $p$. In addition, if $f$ is a rational newform, the condition $\ell \neq p$ can be removed in both cases. The following proposition, which is \cite[Proposition 9.1]{Siksek}, allows us to bound $p$ in some instances by exploiting \eqref{eqn:congruenceconditions}.

\begin{proposition}
    \label{prop:dirbound}
    Suppose that $(x,y,z, p) \in \Z^4$ is a solution to \eqref{eqn:main} with $n = p \ge 7$ prime. Let $f$ be a newform of weight $2$ and level $N'$ as in \eqref{eqn:Np}, with field of coefficients $K_f$ and such that $\overline{\rho}_p(F) \cong \overline{\rho}_p(f)$. Then, for any prime number $\ell \nmid N'$, we define
    \[
        B'_\ell(f) = \Norm_{K_f/\mathbb{Q}}\left((\ell+1)^2-c_\ell^2 \right) \prod_{\substack{|a| < 2\sqrt{\ell} \\ 2 \mid a}} \Norm_{K_f/\mathbb{Q}}(a-c_\ell),
    \]
    and
    \[
        B_\ell(f) = \begin{cases}
            B'_\ell(f) & \text{ if } f \text{ is rational.} \\
            \ell B'_\ell(f) & \text{ otherwise.}
        \end{cases}
    \]
    Then $p \mid B_\ell(f)$.
\end{proposition}

\begin{remark}
    \label{rmk:Bl0}
We remark that, as long as $B_\ell(f) \neq 0$ for all newforms of weight $2$ and level $N'$, we will be able to explicitly find a constant $B_{C, q}$ such that any solutions $(x, y, z, p)$ of \eqref{eqn:main} with $n = p$ prime necessarily satisfy that $p < B_{C, q}$, thereby proving AFLT for \eqref{eqn:main}. Consequently, our aim is to characterise those newforms $f$ for which $B_\ell(f) = 0$.

As stated in the remarks following \cite[Proposition 9.1]{Siksek}, if $B_\ell(f) = 0$, $f$ is necessarily a rational newform and therefore corresponds to an elliptic curve $E$ via the Modularity Theorem. In this instance, we shall write $\overline{\rho}_p(F) \cong \overline{\rho}_p(E)$ to mean $\overline{\rho}_p(F) \cong \overline{\rho}_p(f)$. We also note that in this case $c_\ell = a_\ell(E) := \ell + 1 - \#E(\F_\ell)$.

In addition, if $B_\ell(f) = 0$, we know that $E$ is isogenous to a curve $E'$ with a $\Q-$rational point of order $2$. In this case, it is still true that 
\[\overline{\rho}_p(F) \cong \overline{\rho}_p(E').\]
This fact allows us to prove the following lemma.
\end{remark}

\begin{lemma}
\label{lemma:minimalmodel}
Suppose that AFLT does not hold for \eqref{eqn:main}.
Then there exists an elliptic curve $E$ of conductor $N'$ with a model of the form
\begin{equation}
    \label{eqn:EC}
E: Y^2 = X(X^2+AX+B),
\end{equation}
for certain integers $A, B$ and satisfying
\[\overline{\rho}_p(F) \cong \overline{\rho}_p(E).\]
In addition, the invariants of the minimal model of $E$ are given by
\[c_4 = A^2-3B,\]
\[c_6 = \frac{A(9B-2A^2)}{2},\]
and
\begin{equation}
    \label{eqn:minimaldiscriminant}
\Delta = \frac{B^2(A^2-4B)}{2^8}.
\end{equation}
\end{lemma}

\begin{proof}
    By our discussion in Remark \ref{rmk:Bl0}, Proposition \ref{prop:dirbound} will succeed in bounding $p$ unless there is an elliptic curve $E$ of conductor $N'$ with a point of order $2$ satisfying $\overline{\rho}_p(F) \cong \overline{\rho}_p(E)$. Up to isomorphism, we can assume that $E$ has the following model:
    \begin{equation}
    \label{eqn:initialcurve}
    E: Y^2 = X(X^2+A'X+B'),
\end{equation}
for certain $A', B' \in \mathbb{Z}$. By directly applying the formulas in \cite[Chapter 3]{Silverman}, we find that this model has invariants given by
\[c'_4 = 16({A'}^2-3B'),\]
\[c'_6 = 2^5A'(9B'-2{A'}^2),\]
and
\[\Delta' = 2^4{B'}^2({A'}^2-4B').\]
Suppose that $\ell \neq 2$ is a prime for which the model \eqref{eqn:initialcurve} is not minimal. Therefore, it follows that $\ell^{12} \mid \Delta'$ and $\ell^4 \mid c'_4$. Consequently, $\ell^4 \mid {A'}^2-3B'$ and either
    \[\ell^7 \mid {B'}^2 \quad \text{or} \quad \ell^6 \mid ({A'}^2-4B'),\]
    by the pigeonhole principle. In both cases, we can see that $\ell^4 \mid B'$ and $\ell^2 \mid A'$. Then we can replace $(A', B')$ by 
    \[(A, B) = \left(\frac{A'}{\ell^2}, \frac{B'}{\ell^4}\right)\]
    in \eqref{eqn:initialcurve} and obtain an isomorphic model. After finitely many iterations, this procedure
    will yield a model which is minimal at $\ell$.
    
    Finally, let us consider the case $\ell = 2$. Since $2 \mid c'_4$, $2 \mid \Delta',$ and 2 is a prime of multiplicative reduction for $E$, it follows that \eqref{eqn:initialcurve} cannot be a minimal model at $2$. Consequently, there is another model of $E$ with invariants given by
    \[c''_4 = c'_4/2^4 = {A'}^2-3B',\]
    \[c''_6 = c'_6/2^6 = \frac{A'(9B'-2{A'}^2)}{2},\]
    \[\Delta'' = \Delta'/2^{12} = \frac{{B'}^2({A'}^2-4B')}{2^8}.\]
    If this model is minimal at $2$, we may take $A = A'$ and $B = B'$ and finish the proof. Otherwise, we have that $2^6 \mid c''_6$ and $2^4 \mid c''_4$, and by exploting a similar argument to the case where $\ell \neq 2$, we may see that $2^4 \mid B'$ and $2^2 \mid A'$ and iterately replace $(A', B')$ by $(A'/2^2, B'/2^4)$ until we attain a minimal model.
\end{proof}

\section{An image of inertia argument}
\label{Sec:inertia}
If AFLT does not hold for \eqref{eqn:main}, Lemma \ref{lemma:minimalmodel} gives the existence of an elliptic curve $E$ of conductor $N'$ such that $\overline{\rho}_p(F) \cong \overline{\rho}_p(E)$. In this case, we can see whether the image of the two Galois representation agree for some inertia subgroup of $\Gal(\overline{\Q}/\Q)$. A very useful result in this direction is the fo\-llo\-wing theorem due to Bennett and Skinner, which is \cite[Theorem 13(e)]{Siksek} and follows directly from \cite[Theorem 2.1(d)]{BenS}.

\begin{theorem}(Bennett-Skinner)
\label{thm:imageofinertia}
    Let $(x, y, z, p)$ be a solution to \eqref{eqn:main} with $n=p \ge 7$ prime. Let $F$ be the Frey--Hellegourch curve \eqref{eqn:frey}, and let $E$ be an elliptic curve such that
    \begin{equation*}
\overline{\rho}_p(F) \cong \overline{\rho}_p(E).
\end{equation*}
Then the denominator of the $j-$invariant of $E$ is not divisible by any odd primes $\ell \mid C$ except possibly $\ell = p$.
\end{theorem}

With the use of Theorem \ref{thm:imageofinertia}, we are able to expand on the result of Lemma \ref{lemma:minimalmodel}, giving rise to the following characterisation. For this, we let $\ell$ be a prime number and we let $\nu_\ell(\cdot)$ denote the standard $\ell-$adic valuation.

\begin{proposition}
    \label{prop:inertiacharacterisation}
    Suppose that AFLT does not hold for \eqref{eqn:main} and let $E$ be the elliptic curve given in \eqref{eqn:EC}. Then, for all primes $r \mid C$, we have that $\nu_r(B) = \nu_r(A^2-4B)$.
\end{proposition}

\begin{proof}
    By Lemma \ref{lemma:minimalmodel}, the $j-$invariant of $E$ is given by
    \[j(E) = \frac{c_4^3}{\Delta} = \frac{2^8(A^2-3B)^3}{B^2(A^2-4B)}.\]
    Let $r \mid C$ be a prime satisfying that $\nu_r(B) \neq \nu_r(A^2-4B)$. If $r = p$, we have that $p < C$, and so AFLT holds for \eqref{eqn:main} with $B_{C, q} = C$. If $r \neq p$, $r$ does not divide the denominator of $j(E)$ by Theorem \ref{thm:imageofinertia}. Consequently
    \begin{equation}
        \label{eqn:valuationinequality}
    3\nu_r(A^2-3B) \ge 2\nu_r(B) + \nu_r(A^2-4B).    
    \end{equation}
    Since $\nu_r(B) \neq \nu_r(A^2-4B)$, standard properties of $r-$adic valuations yield that
    \[\nu_r(A^2-3B) = \nu_r((A^2-4B) + B) = \min\{\nu_r(B), \nu_r(A^2-4B)\},\]
    while
    \[2\nu_r(B) + \nu_r(A^2-4B) > 3\min\{\nu_r(B), \nu_r(A^2-4B)\}.\]
    This gives a contradiction with \eqref{eqn:valuationinequality}. Consequently, $\nu_r(B) = \nu_r(A^2-4B)$ and the proposition follows.
\end{proof}

\section{Using Galois theory to provide local information}
\label{Sec:galois}
Our aim in this section is to find conditions under which Proposition \ref{prop:dirbound} can be refined. A key ingredient about the Frey--Hellegouarch curve $F$ that we use in the proof of Lemma \ref{lemma:minimalmodel} is the fact that $F(\Q)$ always has a point of order $2$, $(0,0)$ and, consequently, $F(\F_\ell)$ has a point of order $2$ for all primes $\ell$ of good reduction. For a subset of these primes, it happens that $F(\F_\ell)$ has a subgroup of order $4$ and this fact can be exploited to improve upon Proposition \ref{prop:dirbound} in order to obtain a bound for $p$. 

Our aim is to characterise under what conditions such a prime $\ell$ fails to exist, and we shall do so in this section. In order to do this, we need to prove certain facts about the Frey--Hellegouarch curve $F$, and we do so in the following subsection.

\subsection{Some computations on the Frey curve}
Let $F'$ be any curve which is isogenous to $F$ via a rational isogeny of degree $2^m$ (note that the case $m=0$ means that $F(\Q)$ and $F'(\Q)$ are isomorphic). The aim of this subsection is to show that $F'(\Q)$ can never have a subgroup of order $4$. We shall do that in three steps, corresponding to Lemmas \ref{lemma:nofull2torsion} and \ref{lemma:noorder4}, where we show that this fact is true for $F$, and Lemma \ref{lemma:isogenies}, where we show the same for all $F' \not \cong F$.

\begin{lemma}
    \label{lemma:nofull2torsion}
    The Frey-Helleguarch curve $F$ never has full $2-$torsion over $\mathbb{Q}$.
\end{lemma}

\begin{proof}
    Following \cite[Exercise 3.7]{Silverman} and the expression for $F$ given in \eqref{eqn:frey}, we see that the roots of the $2-$division polynomial of $F$ are given by the expression
    \[2Y + X = 0,\]
    or, equivalently, 
    \[X = -2Y.\]
    Substituting into \eqref{eqn:frey} and simplifying, we get
    \[-8Y^3+CxY^2-\frac{Cz^p}{32}Y = 0.\]
    The root $Y=0$ corresponds to the $2-$torsion point $(0,0)$. If {$F(\Q)$ had additional $2-$torsion points, there would be} other rational roots. This would mean that 
    \[C^2x^2 - Cz^p \ge 0.\]
    But {since $(x, y, z, p)$ is a solution to \eqref{eqn:main}, }we see that
    \[C^2x^2 - Cz^p = -Cq^ky^{2p},\]
    which is clearly negative. Consequently, the only point of order $2$ in $F(\Q)$ is $(0,0)$. 
\end{proof}

\begin{lemma}
    \label{lemma:noorder4}
    The group $F(\mathbb{Q})$ never has a point of order $4$.
\end{lemma}

\begin{proof}
    Suppose that there exists a point $(X_0, Y_0) \in F(\mathbb{Q})$ of order $4$. By Lemma \ref{lemma:nofull2torsion}, the only $\Q-$rational point of order $2$ is $(0,0)$, so it follows that
    \[[2](X_0, Y_0) = (0,0).\]
    Consequently,
    the tangent line to $F$ at $(X_0, Y_0)$ passes through $(0,0)$. {Algebraically, }this condition is equivalent to
    \begin{equation}
        \label{eqn:tangent}
        2(Y_0^2+X_0Y_0) = 3X_0^3 + \frac{Cx-1}{2}X_0^2 + \frac{Cz^p}{64}X_0.
    \end{equation}
    Since $(X_0, Y_0) \in F(\Q)$, the left-hand side of the previous expression can be replaced by
    \[2\left(X_0^3 + \frac{Cx-1}{4}X_0^2 + \frac{Cz^p}{64}X_0\right),\]
    and, consequently, it can be seen that \eqref{eqn:tangent} amounts to
    \[X_0(X_0^2 - Cz^p/64) = 0.\]
    Since $(X_0, Y_0)$ has order $4$, it is clear that $(X_0, Y_0) \neq (0,0)$ and, consequently
    \[X_0^2 = \frac{Cz^p}{64},\]
    but this is not possible since $C$ is squarefree and $\gcd(C, z) = 1$, so $\sqrt{Cz^p}$ is not a rational number.
\end{proof}

These two lemmas show that $F(\Q)$ does not have a subgroup of order $4$. We can use them to prove that the same is true for curves which are isogenous to $F$ via a rational $2^m-$isogeny in the following lemma.

\begin{lemma}
    \label{lemma:isogenies}
    Let $F$ be the Frey--Hellegouarch curve \eqref{eqn:frey} and let $m \ge 1$ be an integer. Let $F'$ be an elliptic curve which is isogenous to $F$ via a rational isogeny of degree $2^m$. Then $F'(\Q)$ does not have a subgroup of order $4$. 
\end{lemma}

\begin{proof}
    By Lemmas \ref{lemma:nofull2torsion} and \ref{lemma:noorder4}, $F(\Q)$ has only one subgroup of order $2^m$ with $m \ge 1$, and this is the subgroup generated by the point $(0,0)$. Consequently, there is only one possible isogenous curve $F'$ to consider. By \cite[Example III.4.5]{Silverman}, this curve has a model given by
    \[F': V^2 = U^3 - \frac{Cx}{2}U^2 - \frac{Cq^ky^{2p}}{16}U,\]
    and so it suffices to see that $F'(\Q)$ does not have a subgroup of order $4$. Firstly, we see that it does not have full $2-$torsion, as this would imply that 
    \[\frac{C^2x^2 + Cq^ky^{2p}}{4} = \frac{Cz^p}{4}\]
    is a rational square. But this is not true since $C$ is squarefree and $\gcd(C, z) = 1$ by assumption.

    Finally, we can see that $F'(\Q)$ has no point of order $4$ by mimicking the approach in Lemma \ref{lemma:noorder4}. Indeed, we recall that $(0,0)$ is the unique rational point of order $4$ in $F'(\Q)$. Consequently, if $(U_0, V_0) \in F'(\Q)$ is a point of order $4$, it follows that $[2](U_0, V_0) = (0,0)$, and, after performing some computations, we find that
    \[U_0^2 + \frac{Cq^ky^{2p}}{16} = 0,\]
    which clearly has no solutions since $Cq^ky^{2p} > 0$. Consequently, $F'(\Q)$ has no subgroup of order $4$.
\end{proof}

With this lemma, we can prove the following corollary, which will be useful in Section \ref{Sec:Galoissieve}.

\begin{corollary}
    \label{cor:Enopointsoforder4}
    Suppose that AFLT does not hold for \eqref{eqn:main}, and let $E$ the elliptic curve given in Lemma \ref{lemma:minimalmodel}. 
    Then no curve isogenous to $E(\Q)$ has a subgroup of order $4$.
\end{corollary} 

\begin{proof}
    Assume for contradiction that some curve in the isogeny class of $E(\Q)$ has a subgroup of order $4$, and let $F$ be the Frey--Hellegouarch curve \eqref{eqn:frey}. By Lemma \ref{lemma:isogenies}, it follows that no curve isogenous to $F$ via a rational $2^m-$isogeny has a subgroup of order $4$. Then, by \cite[Problem I (bis) and Theorem 1]{Katz}, there exists a prime number $\ell$ such that $\ell \nmid N$, $4 \nmid \#F(\F_\ell)$ and $4 \mid \#E(\F_\ell)$. If we define the sets
    \begin{equation}
        \label{eqn:Al}
    A_\ell = \{a \in \mathbb{Z} \mid |a| < 2\sqrt{\ell}, \quad a \equiv \ell+3 \pmod 4\},
    \end{equation}
    and
    \begin{equation}
        \label{eqn:Bl}
    B_\ell = \{b \in \mathbb{Z} \mid |b| < 2\sqrt{\ell}, \quad b \equiv \ell+1 \pmod 4\}, 
    \end{equation}
    then the Hasse--Weil bounds, along with the previous discussion, yield that 
    \[a_\ell(F) \in A_\ell \quad \text{and} \quad a_\ell(E) \in B_\ell.\]
    By \eqref{eqn:congruenceconditions}, we have that
    \begin{equation}
        \label{eqn:divisioncondition}
    p \mid \prod_{\substack{a \in A_\ell \\ b \in B_\ell}} (a-b) \prod_{b \in B_\ell} (b^2-(\ell+1)^2).
    \end{equation}
    Since $A_\ell$ and $B_\ell$ are clearly disjoint, $a - b \neq 0$ for any $a \in A_\ell, b \in B_\ell$. In addition, we also have that $\pm (\ell+1) - b \neq 0$ for any $b \in B_\ell$ by the Hasse--Weil bounds. Thus, \eqref{eqn:divisioncondition} means that $p$ divides a non-zero number. This is a contradiction with the fact that AFLT does not hold and, consequently, it follows that no curve in the isogeny class of $E(\Q)$ has a subgroup of order $4$.
\end{proof}

\subsection{A Galois theory sieve}
\label{Sec:Galoissieve}
After Corollary \ref{cor:Enopointsoforder4}, we are left with the case where $E(\mathbb{Q})$ has no subgroup of order $4$. In order to find a bound $B_{C, q}$ for the exponent in this situation, it is sufficient to find a prime number $\ell \nmid N$ satisfying the following two properties:
\begin{itemize}
    \item The group $F(\mathbb{F}_\ell)$ has a subgroup of order $4$.
    
    \item The group $E(\mathbb{F}_\ell)$ does not have a subgroup of order $4$.
\end{itemize}

In this case, a similar argument to that of the proof of Corollary \ref{cor:Enopointsoforder4} allows to find an upper bound $B_{C, q}$ for $p$, therefore proving AFLT for \eqref{eqn:main}. Consequently, if AFLT does not hold, such a prime $\ell$ cannot exist. In Proposition \ref{prop:galoistheoryconditions}, we find necessary conditions for the non-existence of these primes.

\begin{proposition}
    \label{prop:galoistheoryconditions}
    Suppose that AFLT does not hold for \eqref{eqn:main}, let $F$ be the Frey-Hellegouarch curve \eqref{eqn:frey} and $E$ be the elliptic curve given by Lemma \ref{lemma:minimalmodel}. Let $s \in \{0,1\}$ satisfy $k \equiv s \pmod 2$. Then all primes $\ell \nmid N$ satisfy at least one of the following conditions:
    
    \begin{enumerate}[(i)]
        \item $-Cq^s$ is not a square modulo $\ell$.
        \item $A^2-4B$ is a square modulo $\ell$.
        \item $B$ is a square modulo $\ell$.
    \end{enumerate}
\end{proposition}

The proof of Proposition \ref{prop:galoistheoryconditions} uses the following proposition, which is proved in \cite[Proposition 6.4]{secondpaper} and \cite[Proposition 6.4]{firstpaper}.

\begin{proposition}
    \label{prop:discriminant-trick}
    {Let $E$ be an elliptic curve defined over $\mathbb{Q}$ with discriminant $\Delta$, and let $\ell$ be a prime of good reduction for $E$. Furthermore, assume that $E$ has at least one $\mathbb{Q}-$rational point of order $2$. Then the reduced curve  has full 2-torsion over $\F_\ell$, if, and only if,the reduced discriminant  ${\Delta}$ \text{is a square mod} $\ell$.}
\end{proposition}

\begin{proof}[Proof of Proposition \ref{prop:galoistheoryconditions}]
    Suppose for contradiction that there is a prime $\ell$ not sa\-tis\-fying any of the three conditions (i), (ii) or (iii). By \eqref{eqn:freyminimaldiscriminant}, the discriminant of the Frey--Hellegouarch curve $F$ is
    \[\Delta = -2^{-12}C^3q^k(yz)^{2p},\]
    which, up to multiplication by rational squares, is equivalent to $-Cq^s$.
    Similarly, \eqref{eqn:minimaldiscriminant} yields that, up to multiplication by a rational square, the discriminant of $E$ is equivalent to $A^2-4B$.
    Then, Proposition \ref{prop:discriminant-trick}, along with the fact that conditions (i) and (ii) are not satisfied, yields that $F(\F_\ell)$ has full $2-$torsion while $E(\F_\ell)$ does not. 

    In order to apply the methodology that we outlined at the beginning of this subsection, it remains to see that $E(\mathbb{F}_\ell)$ has no points of order $4$.     
    Let $(x_0, y_0) \in E(\F_\ell)$ be a point of order $4$. Since the only $\mathbb{F}_\ell-$rational point of order $2$ is $(0,0)$, we have that 
    $[2](x_0, y_0) = (0,0)$. By the duplication formula for elliptic curves (see \cite[Group Law Algorithm 2.3]{Silverman}), we have that
    \begin{equation}
        \label{eqn:pointsorder4}
    (x_0, y_0) \in \left\{\left(\sqrt{B}, \pm \sqrt{B}\sqrt{A+2\sqrt{B}}\right), \left(-\sqrt{B}, \pm \sqrt{B}\sqrt{A-2\sqrt{B}}\right)\right\}.
    \end{equation}
    Since condition (iii) is not satisfied, none of these points are $\mathbb{F}_\ell-$rational. Consequently, $E(\F_\ell)$ has no subgroup of order $4$ while $F(\F_\ell)$ does and so, by defining $A_\ell$ and $B_\ell$ as in \eqref{eqn:Al} and \eqref{eqn:Bl} respectively, we may exploit a similar argument to that of the proof of Corollary \ref{cor:Enopointsoforder4} to obtain an upper bound for $p$. Consequently, AFLT holds for \eqref{eqn:main}, which is a contradiction with our hypotheses.
\end{proof}

\section{\texorpdfstring{Finding all possibilities for $E$}{Finding all possibilities for E}}
\label{Sec:possibilities}
In order to prove Theorem \ref{thm:main}, we want to exploit Lemma \ref{lemma:minimalmodel} and Propositions \ref{prop:inertiacharacterisation} and \ref{prop:galoistheoryconditions} to find a list of possibilities for the coefficients $A$ and $A^2-4B$ in the curve $E$. To do this, we need
the following lemma, which is a consequence of Chebotarev's Density Theorem.

\begin{lemma}
    \label{lemma:cebotarev}
    Let $x, y, z \in \Q$ be rational numbers satisfying the following conditions:
    \begin{itemize}
        \item Neither $x$ nor $y$ are rational squares.
        \item The number $z$ is not equivalent to either $x$ or $y$ up to multiplication by rational squares.
    \end{itemize}
    Then there exists a prime number $\ell$ such that $x$ and $y$ are non-squares modulo $\ell$ and $z$ is a square modulo $\ell$.
\end{lemma}

\begin{proof}
    Let $K = \Q(\sqrt{x}, \sqrt{y}, \sqrt{z})$. 
Then our conditions on $x, y$ and $z$ show that there exist an element $\sigma \in \Gal\left(K/\Q\right)$ with $\sigma(\sqrt{x}) = -\sqrt{x},$ $\sigma(\sqrt{y}) = -\sqrt{y}$ and $\sigma(\sqrt{z}) = \sqrt{z}$. 

By Chebotarev's density theorem \cite{Cebotarev} (also stated in Section 3 of \cite{Lenstra}), there is a positive density of primes $\ell$ such that the Frobenius of $K/\Q$ at $\ell$ is equal to $\sigma$. This condition is equivalent to $x$ and $y$ being non-squares modulo $\ell$ and $z$ being a square modulo $\ell$, as desired.
\end{proof}

The possible values for $B$ and $A^2-4B$ will be different depending on the parity of $k$. For simplicity, we separate our argument in two propositions.

\begin{proposition} \label{prop:finalcharacterisationkodd}
    Suppose that $k$ is odd in \eqref{eqn:main}. Suppose furthermore that there does not exist an elliptic curve $E$ given by
    \[E: Y^2 = X(X^2+AX+B),\]
    where $A, B \in \Z$ satisfy one of the following conditions:
    \begin{enumerate}[(A)]
        \item Either we have that 
        \[B = -q^{\gamma_1} \prod_{r \mid C \text{ prime}}r^{\beta_r} ,\]
        \[A^2-4B = 2^{\alpha_2} \prod_{r \mid C \text{ prime}}r^{\beta_r},\]
        with $\alpha_2 > 8$, $\beta_r = 1, 3$ for all $r \mid C \text{ prime}$ and $\gamma_1$ odd, or \\
        
        \item We have that
        \[B = 2^{\alpha_1}\prod_{r \mid C \text{ prime}}r^{\beta_r},\]
        \[A^2-4B = -q^{\gamma_2}\prod_{r \mid C \text{ prime}}r^{\beta_r} ,\]
        with $\alpha_1 > 4$, $\beta_r = 1, 3$ for all $r \mid C \text{ prime}$ and $\gamma_2$ odd.
    \end{enumerate}
    Then AFLT holds for \eqref{eqn:main}.
\end{proposition}

\begin{proof}
    Suppose that AFLT does not hold for \eqref{eqn:main}. Then Lemma \ref{lemma:minimalmodel} yields the existence of a curve $E$ of conductor $N'$ as in \eqref{eqn:Np}. First, let us suppose that $p \neq k$, so that $N' = 2C^2q$. Then it follows that only $2$, $q$ and the primes dividing $C$ can divide the discriminant $\Delta$ given in \eqref{eqn:minimaldiscriminant}, so $B$ and $A^2-4B$ can be supported only on these primes. 

    In addition, {since $2 \mid \mid N'$ and $q \mid \mid N'$,} both $2$ and $q$ are primes of multiplicative reduction for $E$ and therefore divide $\Delta$ while not dividing $c_4$. On the other hand, {$C^2 \mid N'$ and so} all primes $r$ dividing $C$ have additive reduction. Consequently, any such $r$ will divide both $c_4$ and $\Delta$. Since the expressions for $c_4$ and $\Delta$ are given in Lemma \ref{lemma:minimalmodel}, we see that $2$ and $q$ divide precisely one of $A^2-4B$ and $B$ while any prime $r \mid C$ divides both. In addition, by Proposition \ref{prop:inertiacharacterisation}, we have that $\nu_r(B) = \nu_r(A^2-4B)$ for all $r \mid C$ prime. 

    Suppose that there exists a prime $\ell \nmid N$ such that $A^2-4B$ and $B$ are non-squares modulo $\ell$, while $-Cq$ is a square modulo $\ell$. Then Proposition \ref{prop:galoistheoryconditions} gives that AFLT holds for \eqref{eqn:main}, which is a contradiction. By Lemma \ref{lemma:cebotarev}, such a prime $\ell$ will exist unless one of the following conditions are satisfied:
    
    \begin{enumerate}[(a)]
        \item The number $B$ is a rational square.
        
        \item The number $A^2-4B$ is a rational square.
        
        \item The number $B$ is equivalent, up to rational squares, to $-Cq$.
        
        \item The number $A^2-4B$ is equivalent, up to rational squares, to $-Cq$.
    \end{enumerate}
    Let us consider each condition separately. If (a) is satisfied, this would mean that $\sqrt{B} \in \mathbb{F}_\ell$ for every prime number $\ell$. In addition, by the multiplicativity of the Legendre symbol, we have that
    \[\left(\frac{A+2\sqrt{B}}{\ell} \right)\left(\frac{A-2\sqrt{B}}{\ell} \right) = \left(\frac{A^2-4B}{\ell} \right).\]
    Consequently, at least one of the three previous Legendre symbols is equal to $1$. If $A^2-4B$ is a square modulo $\ell$, {the discriminant of the curve $E$ is a square modulo $\ell$ by \eqref{eqn:minimaldiscriminant} and so} $E(\F_\ell)$ has full $2-$torsion by Proposition \ref{prop:discriminant-trick}.


    If either $A+2\sqrt{B}$ or $A+2\sqrt{B}$ are squares modulo $\ell$, at least two of the points of order $4$ in \eqref{eqn:pointsorder4} are defined over $\F_\ell$. In any case, $E(\F_\ell)$ has a subgroup of order $4$ for all $\ell \nmid N'$. By \cite[Theorem I]{Katz}, this means that there exists an elliptic curve $E'$ isogenous to $E$ and such that $E'(\Q)$ has a subgroup of order $4$. Then Corollary \ref{cor:Enopointsoforder4} yields that AFLT holds for \eqref{eqn:main}.
    
    
    
    Suppose now that condition (b) is satisfied. Then the discriminant $\Delta$ of $E$ is a square modulo $\ell$ for every prime number $\ell$. By Proposition \ref{prop:discriminant-trick}, it follows that $4 \mid \#E(\mathbb{F}_\ell)$ for every prime number $\ell \nmid N'$. Then \cite[Theorem I]{Katz} yields that there exists a curve $E'$ isogenous to $E$ and such that $E'(\Q)$ has a subgroup of order $4$. Therefore, Corollary \ref{cor:Enopointsoforder4} gives that AFLT holds for \eqref{eqn:main}. 
    
    \noindent Now, if condition (c) is satisfied, it follows that either
    \begin{align}
    \label{eqn:conditions1}    
    \begin{split}
    B = - 2^{\alpha_1}q^{\gamma_1}\prod_{r \mid C \text{ prime}}r^{\beta_r} , \\
    A^2 - 4B = \prod_{r \mid C \text{ prime}}r^{\beta_r},
    \end{split}
    \end{align}
    with $\alpha_1 > 4$ even, $\beta_r$ odd for all $r \mid C \text{ prime}$ and $\gamma_1 \equiv k \equiv 1 \pmod 2$, or
    \begin{align}
    \label{eqn:conditions2}    
    \begin{split}
    B = -q^{\gamma_1} \prod_{r \mid C \text{ prime}}r^{\beta_r} , \\
    A^2 - 4B = 2^{\alpha_2}\prod_{r \mid C \text{ prime}}r^{\beta_r},
    \end{split}
    \end{align}
    where $\alpha_2 > 8$, $\beta_r$ odd for all $r \mid C \text{ prime}$ and $\gamma_1 \equiv k \equiv 1 \pmod 2$. We remark that $\alpha_1 > 4$ and $\alpha_2 > 8$ because $2$ needs to divide the discriminant $\Delta$ given in \eqref{eqn:minimaldiscriminant}. Now, the set of conditions \eqref{eqn:conditions1} yield that $A^2 < 0$, while the set of conditions \eqref{eqn:conditions2} correspond to case (A) in the statement of the Proposition. 

    Finally, suppose that condition (d) is satisfied. By a similar argument to the one in condition (c), case (B) in the proposition follows. In both situations, we have that the model for the curve $E$ in \eqref{eqn:initialcurve} is isomorphic to one where $\beta_r = 1$ or $\beta_r = 3$, and hence the proposition follows for $p \neq k$. 

    Assume now that $p=k$. In this case, we would have that $N' = 2C^2$. We deal with conditions (a) and (b) exactly as before. We note that, since $q \nmid N'$, it follows that $q \nmid B$ and $q \nmid A^2-4B$ and so conditions (c) and (d) cannot hold.
    
    
\end{proof}

\begin{remark}
    \label{rmk:toughluck}
    A key ingredient in the proof of Proposition \ref{prop:finalcharacterisationkodd} is that $2$ is a prime of multiplicative reduction for the Frey--Hellegouarch curve $F$ and therefore can only divide either $A^2-4B$ or $B$. In addition, as opposed to the rest of primes of additive reduction, it would be impossible to adapt Proposition \ref{prop:inertiacharacterisation} to relate $\nu_2(A^2-4B)$ and $\nu_2(B)$.

    If we allow $z$ to be odd in \eqref{eqn:main}, and as we mentioned at the beginning of Section \ref{Sec:preliminary}, we would need to consider many more Frey--Hellegouarch curves. For some of these, $2$ is a prime of additive reduction and, therefore, it is impossible to use the same arguments as in this paper.
\end{remark}

If $k$ is even, the following proposition gives the possible values of $A^2-4B$ and $B$. Its proof is almost identical to that of Proposition \ref{prop:finalcharacterisationkodd}, and we shall omit it. 

\begin{proposition} \label{prop:finalcharacterisationkeven}
    Suppose that $k$ is even in \eqref{eqn:main}. Suppose furthermore that there does not exist an elliptic curve $E$ given by
    \[E: Y^2 = X(X^2+AX+B),\]
    where $A, B \in \Z$ satisfy one of the following conditions:
    \begin{enumerate}[(A')]
        \item We have 
        \[B = -2^{\alpha_1} \prod_{r\mid C} r^{\beta_r},\]
        \[A^2-4B = q^{\gamma_2} \prod_{r \mid C \text{ prime}} r^{\beta_r},\]
        with $\alpha_1 > 4$ even, $\beta_r = 1, 3$ for all $r \mid C \text{ prime}$ and $\gamma_2 \ge 0$, or
        \item We have 
        \[B = -q^{\gamma_1} \prod_{r \mid C \text{ prime}} r^{\beta_r} ,\]
        \[A^2-4B = 2^{\alpha_2} \prod_{r \mid C \text{ prime}} r^{\beta_r},\]
        with $\alpha_2 > 8$, $\beta_r = 1, 3$ for all $r \mid C \text{ prime}$ and $\gamma_1 \ge 0$ even, or
        \item We have 
        \[B = -\prod_{r \mid C \text{ prime}} r^{\beta_r},\]
        \[A^2-4B = 2^{\alpha_2} q^{\gamma_2} \prod_{r \mid C \text{ prime}} r^{\beta_r},\]
        with $\alpha_2 > 8$, $\beta_r = 1, 3$ for all $r \mid C \text{ prime}$ and $\gamma_2 \ge 0$, or
        \item We have 
        \[B = 2^{\alpha_1} q^{\gamma_1} \prod_{r \mid C \text{ prime}} r^{\beta_r},\]
        \[A^2-4B = -\prod_{r \mid C \text{ prime}} r^{\beta_r},\]
        with $\alpha_1 > 4, \beta_r = 1, 3$ for all $r \mid C \text{ prime}$ and $\gamma_1 \ge 0$, or
        \item We have 
        \[B = 2^{\alpha_1}\prod_{r \mid C \text{ prime}} r^{\beta_r},\]
        \[A^2-4B = - q^{\gamma_2}\prod_{r \mid C \text{ prime}} r^{\beta_r},\]
        with $\alpha_1 > 4, \beta_r = 1, 3$ for all $r \mid C \text{ prime}$ and $\gamma_2 \ge 0$ even, or
        \item We have 
        \[B = q^{\gamma_1} \prod_{r \mid C \text{ prime}} r^{\beta_r},\]
        \[A^2-4B = -2^{\alpha_2} \prod_{r \mid C \text{ prime}} r^{\beta_r},\]
        with $\alpha_2 > 8$ even, $\beta_r = 1, 3$ for all $r \mid C \text{ prime}$ and $\gamma_1 \ge 0$.
    \end{enumerate}
    Then AFLT holds for \eqref{eqn:main}.
\end{proposition}

\section{\texorpdfstring{Proof of Theorem \ref{thm:main}}{Proof of Theorem \ref{thm:main}}}
\label{Sec:theorem1}
In this section, we prove Theorem \ref{thm:main} by studying the values for $B$ and $A^2-4B$ given in Propositions \ref{prop:finalcharacterisationkodd} and \ref{prop:finalcharacterisationkeven}. We shall do this by relating the existence of these curves to the existence of solutions to certain Diophantine equations. This is compiled in the following proposition.

\begin{proposition}
    \label{prop:finalprop}
    Suppose that AFLT does not hold for \eqref{eqn:main}. Then one of the following three alternatives hold:
    
    \begin{enumerate}[(i)]
        \item There is a solution $(t, \gamma, m) \in \Z^3$ to the equation
        \begin{equation*}
        Ct^2 + q^\gamma = 2^m, \text{ with $m > 6$ and $\gamma \ge 0$ with $\gamma \equiv k \pmod 2$,}
    \end{equation*}
    
        \item The exponent $k$ is even, $q \equiv 7 \pmod 8$ and there is a solution $(t, m , \gamma) \in \Z^3$ to the equation
        \begin{equation}
        \label{eqn:keven1}
            Ct^2 + 2^m = q^\gamma, \text{with $m > 6$ even and $\gamma > 0$ odd.}
        \end{equation}
        
        \item The exponent $k$ is even and there is a solution $(t, m, \gamma) \in \Z^3$ to the equation
        \begin{equation}
            \label{eqn:keven2}
            Ct^2 + 1 = 2^m q^\gamma, \text{with $m > 6$ and $\gamma \ge 0$.}
        \end{equation}
    \end{enumerate}
\end{proposition}

\begin{proof}

First, let us suppose that either $k$ is odd or $k$ is even and alternatives $(B')$ or $(E')$ in Proposition \ref{prop:finalcharacterisationkeven} hold. Then we have that either
\[A^2 = 2^2\left(2^{\alpha_2-2}\prod_{r \mid C \text{ prime}}r^{\beta_r} - \prod_{r \mid C \text{ prime}}r^{\beta_r} q^{\gamma_1}\right),\]
or
\[A^2 = 2^{\alpha_1+2}\prod_{r \mid C \text{ prime}}r^{\beta_r} - \prod_{r \mid C \text{ prime}}r^{\beta_r} q^{\gamma_2},\]
where $\alpha_1 > 4$, $\alpha_2 > 8$, $\beta_r = 1, 3$ for all $r \mid C \text{ prime}$, and $\gamma_1, \gamma_2 \ge 0$ with $\gamma_1 \equiv \gamma_2 \equiv k \pmod 2$. Since $C$ is squarefree, both expressions can be rewritten as
\[A^2 = Cw^2(2^m-q^\gamma),
\]
for certain integers $w > 0$, $m > 6$ and $\gamma \ge 0$ with $\gamma \equiv k \pmod 2$. Thus, it follows that
\[2^m - q^{\gamma} = Ct^2,\]
for certain integer $t > 0$. This corresponds to case (i) in the proposition.
Suppose now that $k$ is even and that we are in cases $(A')$ or $(F')$ of Proposition \ref{prop:finalcharacterisationkeven}. Then we have that either
\[A^2 = \prod_{r \mid C \text{ prime}} r^{\beta_r}\left(q^{\gamma_2} - 2^{\alpha_1+2}\right),\]
or
\[A^2 = 2^2\prod_{r \mid C \text{ prime}} r^{\beta_r}\left(q^{\gamma_1} - 2^{\alpha_2-2}\right),\]
with $\beta_r = 1, 3$ for all $r \mid C \text{ prime}$, $\gamma_1, \gamma_2 > 0$, $\alpha_1 > 4$ even and $\alpha_2 > 8$ even. Hence, by similar reasoning, there is a solution $(t, m, \gamma) \in \Z^3$ of 
\begin{equation}
    \label{eqn:dio2}
Ct^2 + 2^m = q^\gamma,
\end{equation}
with $m > 6$ even. Note that, in order for \eqref{eqn:main} to have a solution with $z$ and $k$ even, it follows that $C \equiv 7 \pmod 8$. From \eqref{eqn:dio2}, we see that this implies that $q \equiv 7 \pmod 8$ and $\gamma$ is odd. This proves \eqref{eqn:keven1} and alternative (ii) of the proposition. \\

Finally, let us consider the remaining cases $(C')$ and $(D')$ of Proposition \ref{prop:finalcharacterisationkeven}. In these cases, we have that
\[A^2 = 2^2\prod_{r \mid C \text{ prime}} r^{\beta_r}\left(2^{\alpha_2-2}q^{\gamma_2} - 1\right),\]
\[A^2 = \prod_{r \mid C \text{ prime}} r^{\beta_r}\left(2^{\alpha_1+2}q^{\gamma_1} -1 \right),\]
with $\beta_r = 1, 3$ for all $r \mid C \text{ prime}$, $\alpha_1 > 4$, $\alpha_2 > 8$, $\gamma_1 \ge 0$ and $\gamma_2 \ge 0$. Once more, this is equivalent to the existence of an integral solution $(t, m, \gamma) \in \Z^3$ of
\[Ct^2 + 1 = 2^mq^\gamma,\]
where $m > 6$ and $\gamma \ge 0$, which proves \eqref{eqn:keven2} and alternative (iii) in the proposition.
\end{proof}

Theorem \ref{thm:main} is a direct consequence of the previous proposition, and the proof is immediate.

\begin{proof}[Proof of Theorem \ref{thm:main}]

    In order to see that AFLT holds for \eqref{eqn:main}, it suffices to see that none of the conditions (i), (ii) or (iii) in Proposition \ref{prop:finalprop} are satisfied. 

    By the hypotheses in the Theorem, \eqref{eqn:mainobstruction} does not have any solutions and so alternative (i) does not hold. Similarly, hypotheses (a), (b), (c) and (d) in Theorem \ref{thm:main} directly imply that alternatives (ii) and (iii) in Proposition \ref{prop:finalprop} are not safisfied. Consequently, AFLT holds for \eqref{eqn:main}, as we wanted to show.
    
    



\end{proof}

\section{Checking the conditions in Theorem \ref{thm:main}}
\label{Sec:computation}

In order to computationally verify whether the conditions in Theorem \ref{thm:main} are met, we need to resolve three Diophantine equations (\eqref{eqn:mainobstruction}, \eqref{eqn:obstruction1} and \eqref{eqn:obstruction2}). This can be done by means of the \texttt{Magma} code available in \href{https://github.com/PJCazorla/Asymptotic-Fermat-s-Last-Theorem-for-a-family-of-equations-of-signature--n--2n--2-/tree/main}{https://shorturl.at/hoxW8}, which we will briefly explain in this section.


For \eqref{eqn:mainobstruction}, the methods outlined by the author in \cite{secondpaper} could be used to achieve a complete solution for the more general equation
\begin{equation}
    \label{eqn:primepowers}
Ct^2 + q^\gamma = w^m.
\end{equation}
These methods involve the use of the modular methodology, along with lower bounds for linear forms in three logarithms and the resolution of Thue--Mahler equations. 

However, solving \eqref{eqn:mainobstruction} is, in practice, significantly easier than solving \eqref{eqn:primepowers} since $w$ is restricted to be a power of $2$. We note that \eqref{eqn:mainobstruction} is a particular case of \eqref{eqn:main}, and, therefore, we may then adapt the Frey--Hellegouarch curve \eqref{eqn:frey} to \eqref{eqn:mainobstruction} by setting $x = t$, $y=1$ and $z = 2$. Then we realise that, by \eqref{eqn:conductor}, $F$ has conductor equal to $N = 2C^2q$ if $\gamma \neq 0$ and equal to $N=2C^2$ if $\gamma = 0$. In addition, the minimal discriminant is equal to
\[\Delta = -2^{2m-12}C^3q^\gamma.\]
If $N < 500,000$, we may get all elliptic curves of conductor $N$ from Cremona's tables (\cite{Cremona}) and recover $\gamma, m$ and subsequently $t$, just by inspecting its minimal discriminant. 

If $N \ge 500,000$, the curve is not in Cremona's database. However, we let $m = 3a+b$, where $a \ge 0$ and $b\in\{0, 1, 2\}$. We also let $\gamma = 6c+d$, where $c \ge 0$ and $d \in \{0, \dots, 5\}$. Then it can then be checked that the point $(U, V)$ given by
\[(U, V) = \left(\frac{C\cdot 2^{a+b}}{q^{2c}}, \frac{C^2\cdot 2^b\cdot t}{q^{3c}}\right)
\]
is a rational point on the elliptic curve $E_{b, d}$ given by the expression
\begin{equation*}
E_{b,d}:V^2 = U^3 - C^32^{2b}q^d.
\end{equation*}
Furthermore, it is clear that the only prime which can occur in the denominators of $U$ and $V$ is $q$ and, consequently, $(U, V)$ is a $\{q\}-$integral point. Therefore, it is sufficient to determine all $\{q\}-$integral points on the $18$ curves $E_{b,d}$, where $b = 0,1,2$ and $d = 0,\dots, 5$. In our code, we do this with a combination of \cite[Algorithm 4.2]{equations} and the built-in \texttt{Magma} function for computing $S-$integral points on elliptic curves.


An identical approach can be used for \eqref{eqn:obstruction2}. For this, we let $m = 6a'+b'$, where $a' \ge 0$ and $b' \in \{0, \dots, 5\}$ and we let $\gamma = 3c' + d'$, with $c' \ge 0$ and $d' \in \{0, 1, 2\}$. Then the pair $(U', V')$ given by
\[(U', V') = \left(\frac{C\cdot q^{c'+d'}}{2^{2a'}}, \frac{C^2\cdot q^{d'}\cdot w}{2^{3a'}}\right),\]
is a $\{2\}-$integral point on one of the $18$ Mordell curves $F_{b', d'}$ given by
\[F_{b', d'}: (V')^2 = (U')^3 - C^3\cdot 2^{b'}\cdot q^{2d'}.
\]
These points can be computed in precisely the same way as before. Finally, for \eqref{eqn:obstruction1}, it is sufficient to resolve the Ramanujan-Nagell type equation
\[u^2 + C = Cv,\]
where $u \in \Z$ and $v \in \Z$ is only supported on the primes $2$ and $q$. This can be done by directly employing \cite[Algorithm 6.2]{equations}.

Finally, by combining all the aforementioned techniques, we can finish the proof of Theorem \ref{thm:computations}.

\begin{proof}[Proof of Theorem \ref{thm:computations}]

If there are any solutions $(x, y, z, n)$ to \eqref{eqn:main}, it is elementary to check that $Cq^k \equiv 7 \pmod{8}$ by reducing \eqref{eqn:main} modulo $8$. If $k$ is even, this condition is equivalent to $C \equiv 7\pmod{8}$ while if $k$ is odd, it is equivalent to $Cq \equiv 7\pmod{8}$.

For each suitable pair in the range $1 \le C \le 70$ and $3 \le q < 100$, our \texttt{Magma} program uses the techniques in this section to check whether the conditions in Theorem \ref{thm:main} are satisfied. The results are shown in Table \ref{tab:computations}.
\end{proof}

\end{document}